\documentclass{amsart}%\usepackage{times}
\usepackage{amsmath}
\usepackage{amssymb}
\usepackage{latexsym}
\usepackage{amsfonts}
\usepackage[dvips]{graphicx}
\usepackage{color}
\usepackage{graphicx,epstopdf}

\newtheorem{theorem}{Theorem}

\newtheorem{prop}{Proposition}
\newtheorem{corollary}{Corollary}

\begin{document}
\title
{Whittaker-Hill equation and semifinite-gap Schr\"odinger operators}%
\author{A.D. Hemery}
\address{Department of
Mathematical Sciences,
          Loughborough University,
Loughborough,
          Leicestershire, LE11 3TU, UK}
          \email{A.D.Hemery@lboro.ac.uk}
          
          %    Information for
%second author
         \author{A.P. Veselov}
\address{Department of
Mathematical Sciences,
          Loughborough University,
Loughborough,
          Leicestershire, LE11 3TU, UK
          and
          Moscow State University, Moscow 119899, Russia}

\email{A.P.Veselov@lboro.ac.uk}

\maketitle
\begin{abstract} A periodic one-dimensional Schr\"odinger operator is called semifinite-gap if every second gap in its spectrum is eventually closed. 
We construct explicit examples of semifinite-gap Schr\"odinger operators in trigonometric functions by applying Darboux transformations to the Whittaker-Hill equation. We give a criterion of the regularity of the corresponding potentials and investigate the spectral properties of the new operators.
\end{abstract}

%\small
%\medskip

%{\it Keywords}:  

\medskip

\normalsize

%\newpage

%\tableofcontents

%\newpage

\section{Introduction}
It has been well-known since the work of Floquet and Bloch that the spectrum of the Schr\"odinger operator
$$L=-\frac{d^2}{dx^2}+u(x)$$
with a smooth real periodic potential $u(x+T)\equiv u(x)$ has a band structure. In the 1970s starting with Novikov's famous work \cite{N} a beautiful finite-gap theory, linking the spectral theory with classical algebraic geometry, was developed. It turned out that all operators having only a finite number of gaps in their spectrum can be described explicitly in terms of hyperelliptic Riemann's theta functions \cite{IM}. The simplest one-gap operator is a particular case of the famous Lam\'e operator
$$L=-\frac{d^2}{dx^2}+2\wp(x),$$
where $\wp(x)$ is the Weierstrass elliptic function (shifted by the imaginary half-period to make it non-singular). None of the smooth real periodic finite-gap potentials (except constants)  can be expressed in terms of elementary trigonometric functions.

On the other hand, there are smooth trigonometric potentials, for which every second gap will be eventually closed. We call such potentials {\it semifinite-gap.} A trivial example can be given by any potential, which has a period $T/2$, but is considered as $T$-periodic. One can easily check that for such a potential all the odd gaps are closed. In fact, all the potentials with odd gaps closed are of this form (see Theorem XIII.91 in \cite{RS}).

A more interesting example is given by the so-called {\it Whittaker-Hill} potential
$$u(x)= A \cos 2x + B \cos 4x,$$
with $$A=-4\alpha s, \quad B=-2\alpha^2,$$ 
for any real $\alpha$ and natural $s.$ Namely, it is known (Magnus-Winkler \cite{MW}, Djakov-Mityagin \cite{DM}) that for odd $s=2m+1$ all the even gaps  except the first $m$ are closed and for even $s=2m$ the same is true for odd gaps.

In this paper we construct new explicit trigonometric examples of semifinite-gap potentials by applying the Darboux transformation to the Whittaker-Hill operator. The main issue here is the regularity of the corresponding potentials, which we discuss in detail. 

We note that  all non-constant trigonometric potentials, which are the results of Darboux transformations applied to zero potential, are known to be singular (in contrast to the hyperbolic case, when we have many non-singular multisoliton potentials), so the existence of a large class of regular periodic potentials in elementary functions in the semifinite-gap case seems to be interesting.

We should mention here that for the hyperbolic version of Whittaker--Hill operator with
$u(x) = 4\alpha s \cosh 2x + 2 \alpha^2 \cosh 4x$ the Darboux transformations were considered by Razavy in \cite{R}, but the general regularity question was not addressed in that paper (which is  different in the hyperbolic case anyway). In the framework of quasi-exactly solvable systems the Whittaker-Hill equation has appeared as type X in Turbiner's list \cite{T} and was discussed in details in \cite{FGR}.

\section{Whittaker--Hill equation}
The Whittaker--Hill equation
\begin{equation} \label{W-H}
-\psi'' -(4\alpha s \cos2x + 2\alpha^2\cos4x)\psi = \lambda \psi,
\end{equation}
is an eigenvalue problem for the following $\pi$-periodic Schr\"odinger operator
\begin{equation}
\label{WHo}
L = -D^2- (4\alpha s \cos2x + 2\alpha^2 \cos4x),
\end{equation}
where $D=\frac{d}{dx}.$
It depends on two real parameters $\alpha$ and $s$. Change of the sign of $\alpha$ is equivalent to the shift $x\rightarrow x+\pi/2$ and to the change of sign of $s$, so we can assume without loss of generality that $s\geq 0, \, \alpha \ge 0$. When $\alpha=0$ we have the zero potential. 

This equation appeared in the work of Liapunov \cite{L} and was studied extensively by Whittaker \cite{Whit} (see also Ince \cite{Ince}). Here we follow the terminology of the book \cite{MW}
by Magnus and Winkler. The material of this section is mainly based on this book and on the recent paper  \cite{DM} by Djakov and Mityagin.
 
The substitution
\begin{equation} \label{trans} 
\psi = \varphi e^{\alpha\cos2x}
\end{equation} 
carries \eqref{W-H} into
\begin{equation} \label{transformed W-H}
\varphi '' - 4\alpha \sin2x\varphi ' + [\lambda + 2\alpha^2 + 4(s-1)\alpha \cos2x]\varphi  = 0.
\end{equation}
It can be rewritten as
\begin{equation}
K\varphi  = \nu \varphi ,
\end{equation}
where 
\begin{equation}\label{K}
K  = 
-D^2 + 4\alpha(\sin2x)D-4(s-1)\alpha \cos2x 
\end{equation}
 and $\nu = \lambda + 2\alpha^2$. 

Consider $\pi$-periodic solutions to \eqref{transformed W-H} of the form
\begin{equation} \label{periodicform}
\varphi _{even} = A_0 + \sum_{n=1}^{\infty} A_{2n} \cos2nx, \quad \varphi _{odd} = \sum_{n=1}^{\infty} B_{2n} \sin2nx.
\end{equation}
Substituting \eqref{periodicform} into \eqref{transformed W-H} we end up with the following recurrence relations
\begin{equation}
\nu A_0 + 2\alpha (s+1) A_2 = 0 
\end{equation}
\begin{equation} \label{4eqn}
4\alpha (s-1) A_0 + (\nu - 4) A_2 + 2\alpha (s+3) A_4 = 0 
\end{equation}
\begin{equation} \label{even}
2\alpha (s-2k+1) A_{2k-2} + (\nu - 4k^2) A_{2k} + 2 \alpha (s+2k+1) A_{2k+2} 	=0 \; \;	(k \geq 2)
\end{equation} \\
in the even case and 
\begin{equation}
(\nu - 4) B_2 + 2\alpha (s+3) B_4 = 0 
\end{equation}
\begin{equation} \label{odd}
2\alpha (s-2k+1) B_{2k-2} + (\nu - 4k^2) B_{2k} + 2 \alpha (s+2k+1) B_{2k+2} 	=0 \; \;	(k \geq 2)
\end{equation}
in the odd case. Similarly, for anti-periodic solutions
\begin{equation} \label{antiperiodicform}
\varphi _{even} = \sum_{n=0}^{\infty}C_{2n+1} \cos(2n+1)x, \quad \varphi _{odd} = \sum_{n=1}^{\infty} D_{2n+1} \sin(2n+1)x
\end{equation}
we have the following recurrence relations:
\begin{equation}\label{anti-even1}
(\nu + 2\alpha s - 1 ) C_1 + 2\alpha (s+2) C_3 = 0 
\end{equation}
\begin{equation} \label{anti-even}
2\alpha (s-2k+2) C_{2k-3} + (\nu - (2k-1)^2) C_{2k-1} + 2 \alpha (s+2k) C_{2k+1} 	=0 \; \;	(k \geq 2)
\end{equation} 
and 
\begin{equation} \label{anti-odd1}
(\nu - 2\alpha s - 1) D_1 + 2\alpha (s+2) D_3 = 0 
\end{equation}
\begin{equation} \label{anti-odd}
2\alpha (s-2k+2) D_{2k-3} + (\nu - (2k-1)^2) D_{2k-1} + 2 \alpha (s+2k) D_{2k+1} 	=0 \; \;	(k \geq 2).
\end{equation} 

We will start first with  the periodic case. The following proposition \cite{MW,DM} explains the special role of the integer parameters $s$ in the theory of the Whittaker-Hill equation.

\begin{prop} \label{termination}
Suppose that for some $\lambda$ the equation (\ref{transformed W-H}) has two non-zero solutions  (\ref{periodicform}) with $A_{2n}, B_{2n}$ decaying  exponentially fast as $n \to \infty.$ Then $s$ must be an odd integer. 
\end{prop}

\begin{proof}
%{\bf Proof of lemma \ref{termination}}

We have from \eqref{even} and \eqref{odd} for $k \geq 2$
\begin{eqnarray*}
a_k A_{2k-2} + b_k A_{2k} + c_k A_{2k+2} = 0, \\ 
a_k B_{2k-2} + b_k B_{2k} + c_k B_{2k+2} = 0,
\end{eqnarray*}
where $a_k = 2\alpha (s - 2k +1)$, $b_k = \nu - k^2$ and $c_k = 2\alpha (s+2k+1)$. By multiplying the first equation by $B_{2k}$ and the second by $A_{2k}$ and subtracting, we have
$$
a_k \left|
\begin{array}{cc}
A_{2k-2}	&	A_{2k}	\\
B_{2k-2}	&	B_{2k}
\end{array}
\right|
- c_k \left|
\begin{array}{cc}
A_{2k}	&	A_{2k+2}	\\
B_{2k}	&	B_{2k+2}
\end{array}
\right| = 0.
$$
Introducing $\Delta_k = \left| \begin{array}{cc} A_{2k-2}	&	A_{2k} \\ B_{2k-2}	&	B_{2k} \end{array} \right|$ we have
$$
a_k \Delta_k = c_k \Delta_{k+1}
$$
and hence
\begin{eqnarray*}
\Delta_k	&	=	&	\frac{c_k c_{k+1} \dots c_{k+r}}{a_k a_{k+1} \dots a_{k+r}} \Delta_{k+r+1} \\
		&	=	&	\frac{(s+2k+1)(s+2k+3)\dots(s+2k+2r+1)}{(s-2k+1)(s-2k-1)\dots(s-2k-2r+1)}\Delta_{k+r+1}\\
		&	=	&	\frac{2^{r+1}}{(-2)^{r+1}}\frac{(k+s/2+1/2)(k+s/2+3/2)\dots(k+s/2+r+1/2)}{(k-s/2-1/2)(k-s/2+1/2)\dots(k-s/2+r-1/2)}\Delta_{k+r+1} \\
		&	=	&	(-1)^{r+1} \frac{\Gamma(k+s/2+r+3/2)\Gamma(k-s/2-1/2)}{\Gamma(k+s/2+1/2)\Gamma(k-s/2+r+1/2)} \Delta_{k+r+1},
\end{eqnarray*}
where $\Gamma$ denotes the classical Gamma function \cite{WW} and we have used the functional relation $$\frac{\Gamma(z+r+1)}{\Gamma(z)}=z(z+1)\dots(z+r).$$ 
We rewrite this as
\begin{equation*}
\frac{\Gamma(k+s/2+1/2)}{\Gamma(k-s/2-1/2)} \Delta_{k} = (-1)^{r+1} \frac{\Gamma(k+s/2+r+3/2)}{\Gamma(k-s/2+r+1/2)} \Delta_{k+r+1}. \label{r-equation}
\end{equation*}
Now from Stirling's formula \cite{WW} we have the asymptotic relation
$$
\lim_{t \to +\infty} \left| \frac{\Gamma(t+\rho)}{\Gamma(t-\rho)} t^{-2\rho} \right| = 1
$$
for any fixed $\rho$. Since $A_{2n}$ and $B_{2n}$ decay exponentially fast, so does $\Delta_n$. By setting $t=k+r+1$, $\rho = s/2+1/2$, and letting $r \to \infty$ in \eqref{r-equation} we see that the right hand side goes to zero, which implies that the left hand side (which is independent of $r$) must be zero:
$$
\frac{\Gamma(k+s/2+1/2)}{\Gamma(k-s/2-1/2)}\Delta_k = 0,
$$
or, more explicitly
$$
(k-s/2-1/2)(k-s/2+1/2)\dots(k+s/2-1/2)(k+s/2+1/2) \Delta_k = 0
$$
Hence, we either have $\Delta_k = 0$ for all $k\geq 2,$ or one of the brackets must vanish. 
In the first case it is easy to see that either $\varphi_{even}$ or $\varphi_{odd}$ must be  identically equal to zero. In the second case $s$ is an odd integer. 
\end{proof}

Let now $s=2m+1$ be a positive  odd integer. In that case $a_{m+1} = 2\alpha (s - 2m -1)=0,$ so the infinite tri-diagonal matrices corresponding to the relations 
\eqref{even} and \eqref{odd} become reducible. Introduce the corresponding finite-dimensional parts, corresponding to the terminating sequences $A_{2n}=B_{2n}=0$ for $n > m:$
\begin{equation}
\label{ev}
K^{0}_m = 
\begin{pmatrix}
		b_0	&	c_0	&	0	&	\dots		&	\dots	&	0 	\\
		a_1	&	b_1	&	c_1	&	0		&	\dots	&	0	\\
		0	&	a_2	&	b_2	&	c_2		&	\dots	&	0	\\
		\vdots & \ddots	& \ddots 	& 	\ddots	&	\ddots	&	\vdots	\\
		0	&	\dots	&	0	&	a_{m-1}	&	b_{m-1} & c_{m-1} \\
		0	&	\dots & \dots	&	0		&	a_{m}	&	b_{m} \\
\end{pmatrix}
\end{equation}
and
\begin{equation}
\label{od}
K^{1}_m = 
\begin{pmatrix}
		b_1	&	c_1	&	0	&	\dots		&	\dots	&	0 	\\
		a_2	&	b_2	&	c_2	&	0		&	\dots	&	0	\\
		0	&	a_3	&	b_3	&	c_3		&	\dots	&	0	\\
		\vdots & \ddots	& \ddots 	& 	\ddots	&	\ddots	&	\vdots	\\
		0	&	\dots	&	0	&	a_{m-1}	&	b_{m-1} & c_{m-1} \\
		0	&	\dots & \dots	&	0		&	a_{m}	&	b_{m} \\
\end{pmatrix}
\end{equation}
where  $b_k = \nu - 4k^2$, $c_k = 2\alpha (s+2k+1)$ for all $k$, $a_k = 2\alpha (s - 2k +1)$
for $k\geq 2$ and $a_1 = 4\alpha(s-1).$

Let 
$$\delta^0_m(\nu,\alpha) = \det K^0_m, \quad \delta^1_m(\nu,\alpha) = \det K^1_m$$
and $$S_0(\alpha)=\{\nu: \delta^0_m(\nu, \alpha)=0\}, \quad S_1(\alpha)=\{\nu: \delta^1_m(\nu, \alpha)=0\}$$
be the corresponding eigenvalues. 

Since $\alpha>0$ all the off-diagonal entries of the matrices $K^0_m$ and $K^1_m$ are positive. From the general theory of tri-diagonal (Jacobi) matrices (see e.g. \cite{Sze}) it follows that all the eigenvalues $\nu$ are real and distinct. Moreover, these sets $S_0$ and $S_1$ are {\it interlacing}: between any two zeroes of $S_0$ there is exactly one zero of $S_1.$ Although this fact is well-known for the convenience of the reader we give the  proof in the Appendix 1. 

{\it Remark.} In modern terminology $S=S_0 \cup S_1$ corresponds to the {\it solvable} part of the spectrum of the Whittaker--Hill operator, which in the case when $s$ is an integer belongs to the class of {\it quasi-exactly solvable} operators \cite{T,Ush}. In the theory of such operators the polynomials $\delta^0_m$ are sometimes called {\it Bender-Dunne} polynomials \cite{BD}.

A crucial observation due to Magnus-Winkler \cite {MW} (see also \cite{DM})  is that the solvable part coincides with the simple part of the periodic spectrum. More precisely, we have the following 

\begin{prop} \label{multiplicity}
Let $s = 2m +1$, $m \in \mathbb{Z}_+.$ If $\nu$ belongs to the periodic spectrum of the Whittaker-Hill equation and has multiplicity $1,$ then $\nu \in S.$
\end{prop}

\begin{proof}
Let us assume that $\nu \notin S$ and that we have a non-zero even periodic solution of the form
$$
\varphi = A_0 + \sum_{k=1}^{\infty} A_{2k} \cos 2kx.
$$
Since $\nu \notin S$ the coefficient $A_{2m+2} \neq 0$ (otherwise, we have $A_{2k} = 0$ for all $k > m$ and thus $\nu \in S_0$). We now construct a second, odd periodic solution for the same $\nu.$ First we set $B_{2n} = A_{2n}$ for $n > m$, which is OK since the recurrence relations \eqref{even} and \eqref{odd} agree for $k \geq 2.$ Now, since $\nu \notin S_1$ the $m\times m$ matrix $K^1_0$ is invertible. This means that we can reconstruct uniquely the beginning of the sequence 
$B=(B_2, B_4, \dots, B_{2m})$ as
\begin{equation}
B = (K^1_m)^{-1} 
\left(
\begin{array}{c}
0 \\
\vdots \\
0 \\
-c_m A_{2m+2}
\end{array}
\right)
\end{equation}
This means that we have independent periodic solutions with the same $\nu$, so $\nu$ has multiplicity $2$. Contradiction means that $\nu \in S_0$ in that case. A similar argument in the case of odd periodic solution shows that $\nu \in S_1.$
\end{proof}

Similar results are true for even $s=2m$ and anti-periodic spectrum. From  (\ref{anti-even1}),(\ref{anti-even}),(\ref{anti-odd1}),(\ref{anti-odd}) the corresponding eigenvalues coincide with the eigenvalues of the following matrices:  
\begin{equation}
\label{sev0}
K^{\pm}_m = 
\begin{pmatrix}
		b_1^{\pm}	&	c_1	&	0	&	\dots		&	\dots	&	0 	\\
		a_2	&	b_2	&	c_2	&	0		&	\dots	&	0	\\
		0	&	a_3	&	b_3	&	c_3		&	\dots	&	0	\\
		\vdots & \ddots	& \ddots 	& 	\ddots	&	\ddots	&	\vdots	\\
		0	&	\dots	&	0	&	a_{m-1}	&	b_{m-1} & c_{m-1} \\
		0	&	\dots & \dots	&	0		&	a_{m}	&	b_{m} \\
\end{pmatrix}
\end{equation}
where   $c_k = 2\alpha (s+2k),$ $a_k = 2\alpha (s - 2k+2)$for all $k$, 
$b_k= \nu-(2k-1)^2$ for $k\geq 2$ and $b_1^{\pm} = \nu-1 \pm 2\alpha s.$
The corresponding sets $S^{\pm}$ are also interlacing (see Appendix 1). We denote all the eigenvalues from the set $S= S^+ \cup S^-$ in increasing order as $\nu_1, \dots, \nu_{2m}.$

The degeneracy of the eigenvalues $\nu \notin S$ is called the {\it coexistence} property in \cite{MW}.
This leads to the following spectral property of the Whittaker--Hill operator. For the general theory of periodic Schr\"odinger operators we refer to the classical Reed-Simon book \cite{RS}.

Let us choose a basis $\psi_1(x), \, \psi_2(x)$ of the solution space of the Whittaker--Hill equation (\ref{W-H}) and define the corresponding {\it monodromy matrix} by
\begin{equation*}
\left( \begin{array}{c}
\psi_1(x+\pi) \\
\psi_2(x+\pi)
\end{array} \right)
=
M(\lambda)
\left( \begin{array}{c}
\psi_1(x) \\
\psi_2(x)
\end{array} \right).
\end{equation*}
Taking wronskians of both sides we see
 that $\det M(\lambda) = 1.$
 The trace of the monodromy matrix $\Delta(\lambda)=tr M(\lambda)$  is independent of the choice of the solutions $\psi_1, \, \psi_2$ and sometimes is called {\it Hill's discriminant} (in Russian literature the term {\it Liapunov function} is also used).
The {\it Floquet multipliers} $\mu_1$,$\mu_2$ are the eigenvalues of $M(\lambda),$ which satisfy the characteristic equation
\begin{equation*}
\mu^2 - \Delta(\lambda)\mu + 1 = 0.
\end{equation*}
The continuous spectrum of $L$ consists of the segments of $\lambda \in \mathbb R$ such that corresponding solutions are bounded on the real line. This happens when $|\mu_i|=1,$ or equivalently
$$|\Delta(\lambda)|\leq 2.$$
The complement is the union of intervals called {\it gaps}. They correspond to the part of graph of $\Delta(\lambda)$ lying outside of the spectral strip $-2 \leq \Delta \leq 2.$

In general (for example, for the Mathieu operator with $u=A \, \cos \, 2x$) this graph intersects each line $\Delta = \pm 2$ during each wave of oscillation (see Fig. 1, left). In the finite-gap case starting from some point the graph starts to touch these lines but not intersect them (see the one-gap case $u(x) = 2\wp (x)$ on Fig.1, right).

In the semifinite-gap case this happens only for one of these lines. For example, for the Whittaker--Hill operator  the graph eventually touches the periodic line $\Delta=2$ for odd $s$ and  the anti-periodic line $\Delta=-2$ for even $s$ (see Fig.2). The case of small $\alpha$ (see below) shows that the open gaps are precisely the first ones; by continuity arguments this is true for all real $\alpha.$

\begin{figure}
\centerline{ \includegraphics[width=6cm]{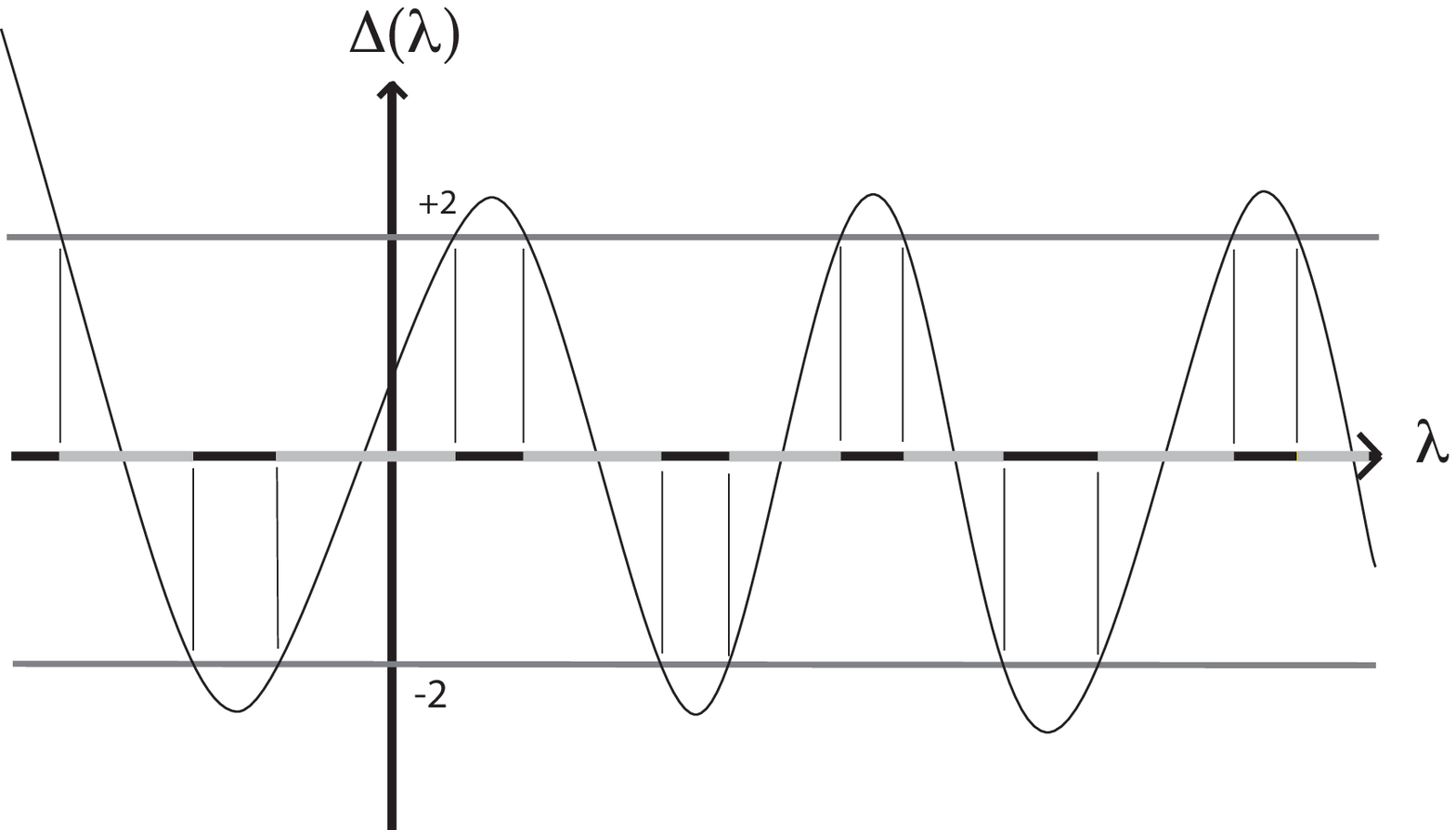} \hspace{10pt} \includegraphics[width=6cm]{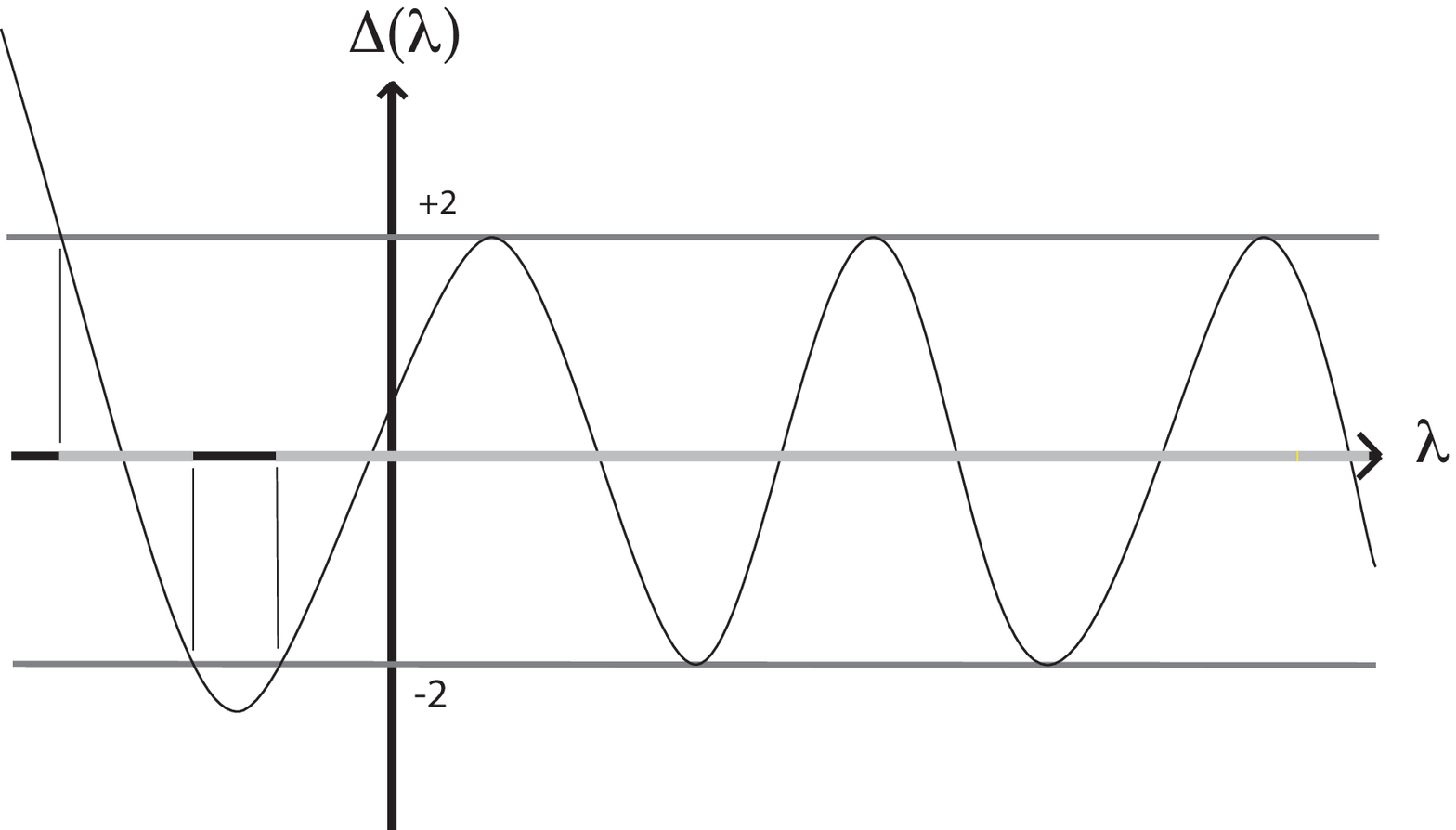} }
\caption{Graphs of the Hill's discriminant. Left: generic periodic case. Right: one gap case} \label{generic}
\end{figure}

\begin{figure}
\centerline{ \includegraphics[width=6cm]{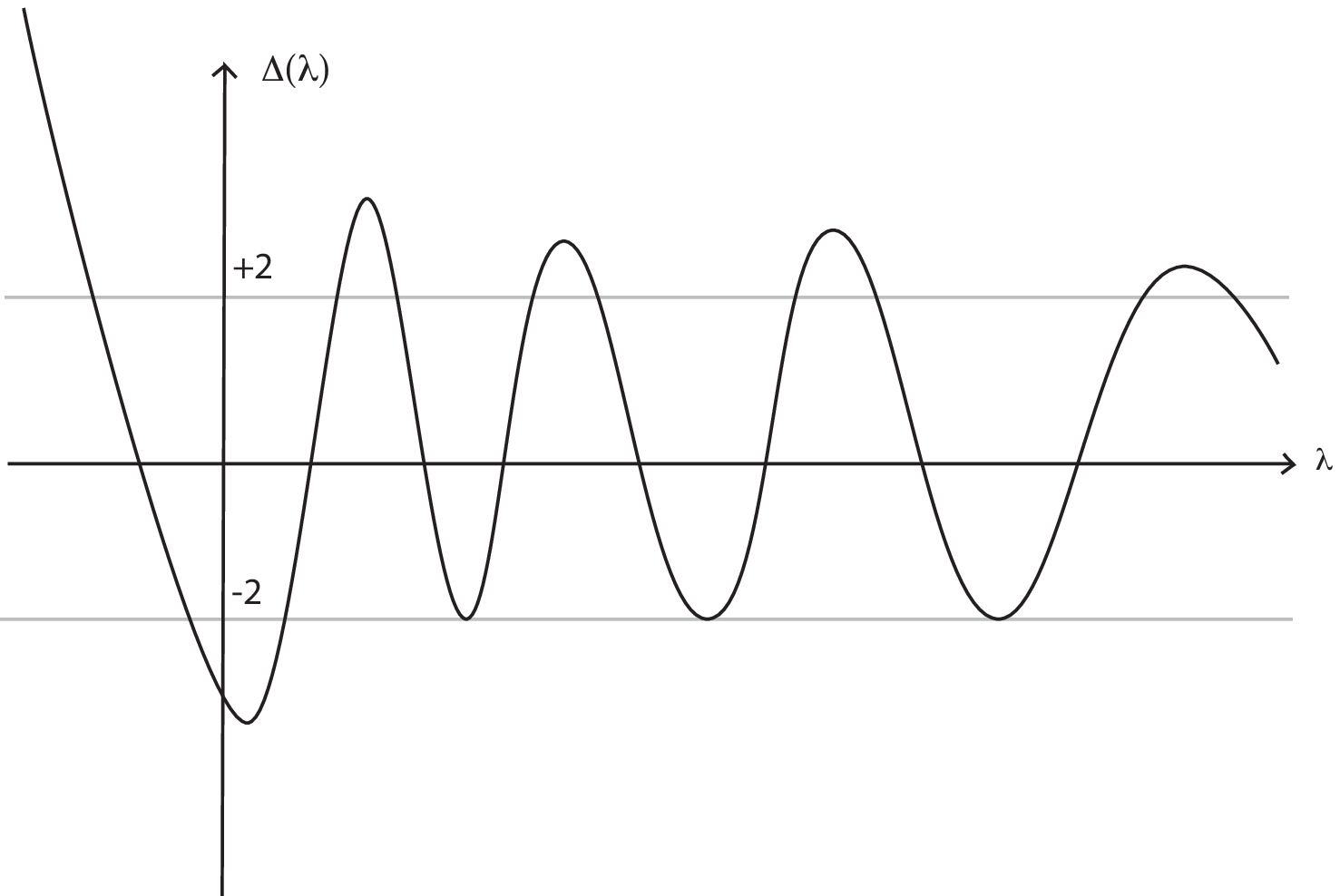} \hspace{10pt} \includegraphics[width=6cm]{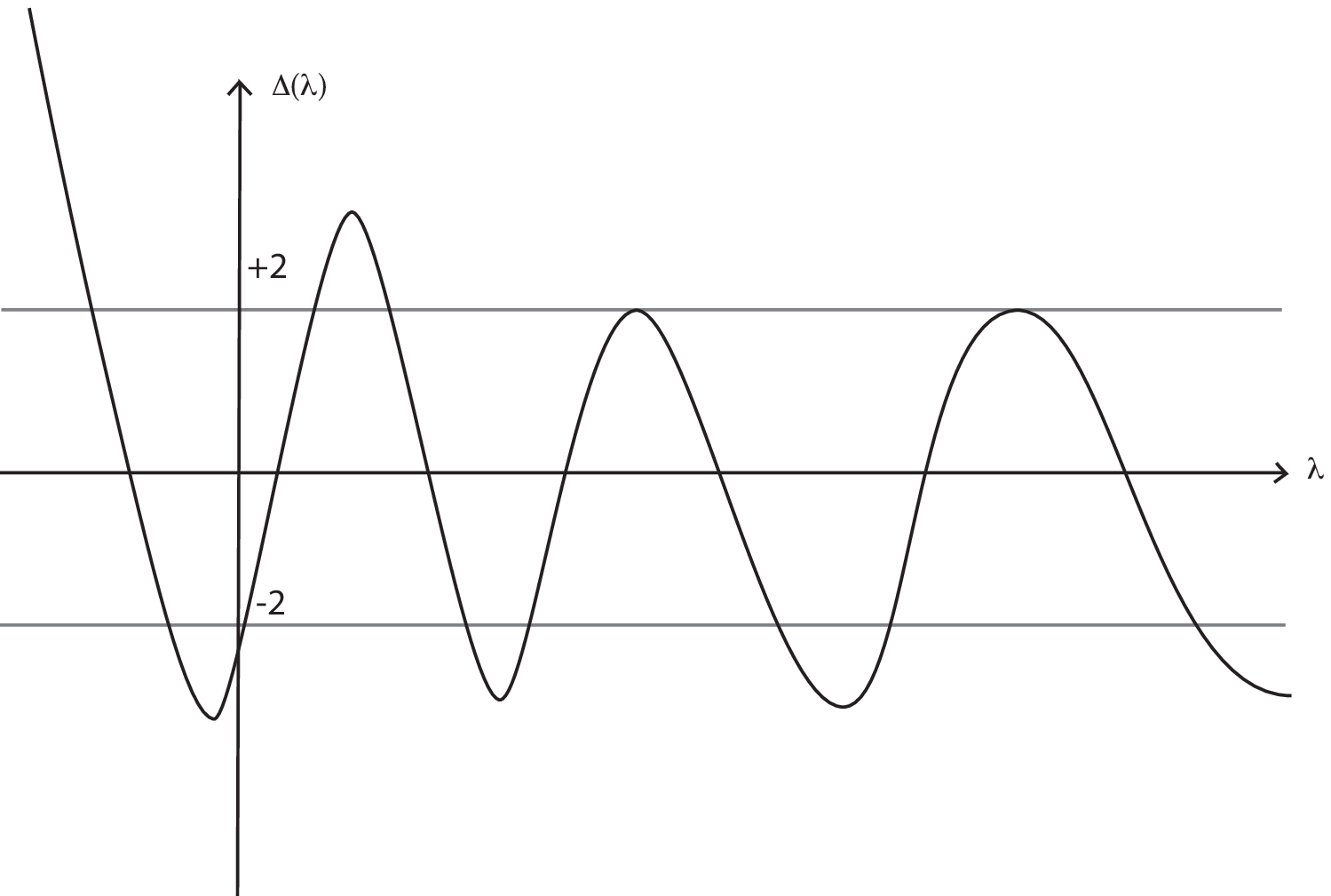} }
\caption{Hill's discriminant in the semifinite-gap case. Left: $s=2$. Right: $s=3.$} \label{semi}
\end{figure}

We summarize the results of this section as follows.
\begin{theorem} \cite{MW, DM} \label{thm}
The Whittaker-Hill operator (\ref{WHo})  has all even gaps closed except the first $m$ when $s=2m+1$ and all odd gaps closed except the first $m$ when $s=2m.$
\end{theorem}

Now we are going to show that there are many other semifinite-gap operators with potentials which can be expressed in terms of trigonometric functions. The idea is standard: to apply the Darboux transformation \cite{Darb}. The main problem is how to get non-singular potentials. 
We are going to see first what happens for small $s$ and $\alpha.$

\section{Examples and  the small $\alpha$ case}

In the first non-trivial case $s=1$ the corresponding operator (\ref{K}) simplifies to 
$$K  = 
-D^2+ 4\alpha \sin2xD$$
and clearly has an eigenfunction $\varphi_0 \equiv 1$ with the eigenvalue $\nu=0.$ The corresponding Whittaker--Hill operator 
$$L= -D^2 -(4\alpha \cos2x + 2\alpha^2\cos4x)$$
has an explicit ground state
$$\psi_0=e^{\alpha\cos2x}$$ with the eigenvalue $\lambda = -2\alpha^2.$

When $s=2$ we have
$$K  = 
-D^2+ 4\alpha \sin2x D-4\alpha \cos 2x.
$$
One can easily check that it has $\varphi_1 = \cos x$ and $\varphi_2 = \sin x$ as the eigenfunctions with $\nu_1=1-4\alpha$ and $\nu_2= 1+4\alpha$ respectively.

In the $s=3$ case  we have the matrix
\begin{equation*}
K^{0}= \left(
\begin{array}{ccc}	
		\nu 	& 8 \alpha		\\
		8 \alpha		          & \nu -4		\\
\end{array}
\right)
\end{equation*}
so we have
$$\delta^0=\nu(\nu-4)-64\alpha^2,\quad \delta^1=\nu-4$$
and the following table
\medskip
\noindent
\begin{center}
\begin{tabular}{ | c  | c | c  | c |}
\hline
\quad & Eigenvalue & \quad & Eigenfunction \\
\hline
$\nu_0$ & $2 ( 1 - \sqrt{1 + 16\alpha^2})$ &  $\varphi_0$ &  $1 + \frac{\sqrt{1+16\alpha^2}-1}{4\alpha}\cos2x$ \\
$\nu_1$ & 4 & $\varphi_1$ & $\sin 2x$ \\
$\nu_2$ & $2 ( 1 + \sqrt{1 + 16\alpha^2})$ & $\varphi_2$ & $\frac{1 - \sqrt{1+16\alpha^2}}{4\alpha} + \cos2x$ \\
\hline 
\end{tabular} 
\end{center}
 For small $\alpha$ we have
 $$\nu_0 \approx -16\alpha^2, \quad \varphi_0 \approx 1+ 2\alpha \cos 2x,$$
 $$\nu_1=4, \quad \varphi_1 = \sin 2x,$$
 $$\nu_2 \approx 4+16\alpha^2, \quad \varphi_0 \approx -2\alpha +  \cos 2x.$$

When $s=4$ the matrices are
\begin{equation*}
K^{\pm}= \left(
\begin{array}{ccc}	
		\nu - 1 \pm 8\alpha	& 12 \alpha		\\
		4 \alpha		          & \nu -9		\\
\end{array}
\right)
\end{equation*}
The eigenvalues are the zeros $\nu=\nu^{\pm}_{1,2}$ of $\delta^{\pm}=\det K^{\pm}:$
$$\delta^{\pm}(\nu) = (\nu-1\pm 8\alpha)(\nu-9)-48\alpha^2=0.$$
The corresponding eigenfunctions are respectively
$$
\varphi_{1,2}^{+}= (\nu_{1,2}^{+} -9) \cos x - 4\alpha \cos 3x,
$$
$$
\varphi_{1,2}^{-}= (\nu_{1,2}^{-} -9) \sin x - 4\alpha \sin 3x.
$$

When $\alpha \rightarrow 0$ we have the first $s$ periodic levels of the corresponding operator
$L_0= -D^2:$
\begin{eqnarray*}
\nu_0=0, \quad \varphi_0 & = & 1 \\
\nu_1=4, \quad \varphi_1 & = & \sin2x \\
\nu_2=4, \quad \varphi_2 & = & \cos2x \\
\vdots &\\
\nu_{2m-1}=4m^2, \,\, \varphi_{2m-1} & = & \sin2mx \\
\nu_{2m}=4m^2, \quad \varphi_{2m} & = & \cos2mx
\end{eqnarray*}
when $s=2m+1$ and
\begin{eqnarray*}
\nu_1=1, \quad \varphi_1 & = & \sin x \\
\nu_2=1, \quad \varphi_2 & = & \cos x \\
\vdots &\\
\nu_{2m-1}=(2m-1)^2, \,\, \varphi_{2m-1} & = & \sin (2m-1)x \\
\nu_{2m}=(2m-1)^2, \quad \varphi_{2m} & = & \cos (2m-1)x
\end{eqnarray*}
for $s=2m.$ 

Note that since we denote the eigenvalues in increasing order it follows from the interlacing property that  for odd $s$ the even eigenvalues correspond to the even eigenfunctions and odd eigenvalues to odd eigenfunctions. For even $s$ the opposite is true: odd eigenvalues correspond to even eigenfunctions and vice versa (see Appendix 1).

We are going to start with this limiting case to see the properties of the wronskians of the corresponding eigenfunctions. 

\section{Darboux transformed Whittaker--Hill operators}

We recall the following classical construction going back to Darboux \cite{Darb}.
Let $\psi_1,\dots, \psi_n$ be a set of eigenfunctions of the Schr\"odinger operator
$$L= -D^2+u(x)$$ with some eigenvalues $\kappa_1,\dots, \kappa_n$:
$$-\psi_j'' + u(x) \psi_j = \kappa_j \psi_j,\,\, j=1,\dots, n.$$
The corresponding {\it Darboux transformation} of $L$ is the operator
$\tilde L= -D^2+\tilde u(x),$ where
\begin{equation}
\label{Dar}
\tilde u = u -2D^2 \log W(\psi_1, \dots, \psi_n),
\end{equation}
where $W(\psi_1, \dots, \psi_n)$ is the wronskian of the functions $\psi_1, \dots, \psi_n.$
The key property of the new operator is that the generic solutions of the corresponding Schr\"odinger equation
$$-\tilde \psi'' + \tilde u(x) \tilde \psi = \lambda \tilde \psi$$
can be explicitly given in terms of the solutions of the initial equation
$$-\psi'' + u(x) \psi = \lambda \psi$$
by the Crum formula \cite{Crum}:
\begin{equation}
\label{Crum}
\tilde\psi = \frac{W(\psi, \psi_1, \dots, \psi_n)}{W(\psi_1, \dots, \psi_n)}.
\end{equation}
This works well for all $\lambda \neq \kappa_j, \,\, j=1,\dots n.$ 

We will need also the following information about eigenfunctions for $\lambda=\kappa_j.$ When $n=1$ the inverse
\begin{equation}
\label{inv}
\tilde \psi_1 = \frac{1}{\psi_1}
\end{equation}
 is a solution of the transformed equation with the eigenvalue $\lambda=\kappa_1.$ For $n=2$ one can check that
\begin{equation}
\label{two}
\tilde \psi_1 = \frac{\psi_2}{W(\psi_1, \psi_2)}, \quad \tilde \psi_2 = \frac{\psi_1}{W(\psi_1, \psi_2)}
\end{equation}
satisfy the transformed equation with $\lambda=\kappa_1$ and $\lambda=\kappa_2$ respectively.

Consider now the Whittaker--Hill operator with 
$$u=-(4\alpha s \cos 2x +2\alpha^2 \cos 4x)$$ with odd $s=2m+1$ and the corresponding set of the explicit eigenfunctions $\psi_0, \dots, \psi_{2m}$ from the previous section. Let $I=\{i_1,\dots, i_k\}$ be the subset of $\{1, \dots, 2m\}$ and 
$$W_I=W(\psi_{i_1}, \dots, \psi_{i_k})$$ be the corresponding wronskian.
Let us introduce also the wronskian of the corresponding functions $\varphi_j=\psi_j e^{-\alpha\cos 2x}:$
$$V_I(x) = W(\varphi_{i_1}, \dots, \varphi_{i_k}) = e^{-k\alpha\cos 2x}W_I.$$ 

We are looking for the subsets $I$ such that the corresponding Darboux transformations
\begin{equation}
\label{Darb}
\tilde u = u -2D^2 \log W_I = 4\alpha (2k-s) \cos 2x -2\alpha^2 \cos 4x -2D^2 \log V_I
\end{equation}
are non-singular. This clearly happens if and only if the wronskian $W_I$ has no real zeros on the whole line. 

We claim that this holds for the following subsets. First of all we call a cluster any pair $\{2k-1, 2k\}$ as well as a single element set $\{0\}.$
By definition, the {\it cluster subsets} $I \subset \{0,\dots, 2m\}$ are those which consist of several clusters.
In other words, the subset $I$ is called a cluster when it contains an element $2k-1$ if and only if it contains $2k$ for all $k=1, \dots, m.$ This applies also to the case of even $s=2m$, except that in that case
there is no single element clusters since $I \subset \{1,\dots, 2m\}.$

\begin{theorem} \label{theo}
For any cluster subset $I$ the corresponding wronskian $W_I(x)$ has no zeros on the real line.
\end{theorem}

\begin{proof}
First of all this is obviously true for $\alpha=0$ (and hence, for small $\alpha$). Indeed, for $I=\{2k-1,2k\}$ 
we have in this case
\begin{eqnarray*}
W_I  & = & \left|
\begin{array}{ccc}	
		\sin2kx			&		\cos2kx		\\
		(2k)\cos2kx		&	(-2k)\sin2kx		\\
\end{array}
\right| \\
& = & (-2k)[\sin^22kx + \cos^22kx] = -2k,
\end{eqnarray*}
which is a constant, hence clearly non-zero for positive $k$. The same can be shown for the general cluster case. 

Now we claim that the cluster wronskian $W_I$ remains non-vanishing anywhere for all (not necessarily small) real values of $\alpha.$  We will prove this by induction on the number $k$ of elements in $I$. 

If $k=1$ then $I=\{0\}$ and $W_I=\psi_0.$ But $\psi_0$ is the ground state and therefore it has no zeroes by a general theorem (see e.g. \cite{RS}).

If $k=2$ then $I=\{2k-1,2k\}$ and $W_I=W(\psi_{2k-1},\psi_{2k}).$
 We know that $W_I (x; \alpha)$ is a smooth function of both $x$ and $\alpha$ and that it has no zeros for small $\alpha.$  Suppose that there is a positive value of $\alpha$ for which $W_I$ has at least one zero and consider the minimal value $\alpha^*$ of $\alpha$ for which is true. By definition, for $\alpha = \alpha^*$  there exists real $x=x_0$ such that
$W_I (x_0) = 0$ and $W_I' (x_0) = 0.$
We claim that this is impossible. Indeed, we have
\begin{equation*}
\label{W'}
W_I' = (\psi_{2k-1} \psi'_{2k} - \psi_{2k}\psi'_{2k-1}) ' = \psi_{2k-1}	\psi''_{2k}	 - \psi_{2k}\psi''_{2k-1} 
= (\lambda_{2k-1} - \lambda_{2k}) \psi_{2k-1} \psi_{2k}.
\end{equation*}
We see that $W_I' (x_0) =0$ implies that either $\psi_{2k-1}(x_0) = 0$ or $\psi_{2k}(x_0) = 0$. Assume without loss of generality that $\psi_{2k}(x_0) = 0.$ Then from $W_I(x_0)=0$ we must have either $\psi_{2k}'(x_0) = 0$ or  $\psi_{2k-1}(x_0)=0$. The first case is impossible since  then $\psi_{2k}$ must be identically zero as a solution of second order linear differential with zero initial data at $x_0.$ 
Suppose that $\psi_{2k-1}(x_0)=0$, which means that $\psi_{2k-1}$ and $\psi_{2k}$ have a common zero $x_0.$ We claim that this is impossible because of the following classical result.

%\noindent
{\bf Sturm's Theorem \cite{Ince}}. 
{\it Let $u(x)$, $v(x)$ be smooth functions on the interval $[a,b]$ such that
\[ u''(x) = G_1(x) u(x) \]
\[ v''(x) = G_2(x) v(x) \]
with $G_1 > G_2$ on $[a,b]$. Then between any two consecutive zeros of $u$ there exists a zero of $v$.}

\medskip
Applying this to the Whittaker-Hill equation $\psi'' = (u(x) - \lambda)\psi$ we see that since $\lambda_{2k-1} < \lambda_{2k}$ we must have a zero of $\psi_{2k-1}$ between any two consecutive zeros of $\psi_{2k}.$ We know that for $\alpha=0$ both $\psi_{2k}=\cos2kx$ and $\psi_{2k-1}=\sin2kx$ have $2k$ zeros on $[0,\pi)$. When $\alpha$ is increasing from zero $\psi_{2k}$ and $\psi_{2k-1}$ must maintain the number of zeros since otherwise there must be a value of $\alpha$ for which the corresponding solution $\psi$ has a multiple zero, which is impossible for a second order differential equation. But this, in combination with the coincidence of the zeros of $\psi_{2k}$ and $\psi_{2k-1},$ contradicts the interlacing property guaranteed by Sturm's theorem. This proves the one cluster case.

Suppose now we have cluster $I=\{2k-1, 2k, 2l-1, 2l\}.$ Then we can do the corresponding Darboux transformation in two steps. First we apply it to the cluster $I_k=\{2k-1, 2k\}.$ We know already that
the result is a non-singular periodic Schr\"odinger operator having the transformed eigenfunctions $\tilde\psi_{2l-1}, \tilde\psi_{2l}.$ Now we
repeat the same arguments for these eigenfunctions to show that $W(\tilde\psi_{2l-1}, \tilde\psi_{2l})$ has no real zeros. But 
$$W(\tilde\psi_{2l-1}, \tilde\psi_{2l})=\frac{W(\psi_{2k-1}, \psi_{2k}, \psi_{2l-1}, \psi_{2l})}{W(\psi_{2k-1}, \psi_{2k})},$$  so this implies that $W_I= W(\psi_{2k-1}, \psi_{2k}, \psi_{2l-1}, \psi_{2l})$ has no real zeros as well. Continuing in this way we prove the theorem.
\end{proof}

Thus for $s=2m+1$ we have constructed
$2^{m+1}-1$ new potentials corresponding to the $2^{m+1}-1$ cluster sets.
However, some of them are equivalent. For example,
the largest cluster set $I=\{0,\dots,2m\}$ corresponds to the potential $$
L^*= -D^2 + 4\alpha s \cos 2x -2\alpha^2 \cos 4x,
$$
which is just the same potential (\ref{WHo}) shifted by $x \rightarrow x+\pi/2.$ Indeed, in that case
$$V_I=W(\varphi_0,\dots, \varphi_{2m}) = W(1, \sin 2x, \cos 2x, \dots, \sin 2mx, \cos 2mx) \equiv const$$
since the linear span of $\varphi_0,\dots, \varphi_{2m}$ coincides with the whole space of trigonometric polynomials up to degree $2m.$ This is similar to the duality in the sextic case considered in \cite{GV}.

This allows us to consider only the cluster subsets $S \subset \{1,2,\dots, 2m\}$ in both $s=2m+1$ and $s=2m$ cases.
The corresponding Schr\"odinger operators have the same continuos spectrum, but different auxiliary
eigenvalues $\gamma_i,$ corresponding to the Dirichlet eigenvalue problem
$$\psi(0)=\psi(\pi)=0.$$ Since all our potentials are even these eigenvalues coincide with one of the ends of the corresponding gap. For the initial Whittaker--Hill operator (\ref{WHo}) with $s=2m+1$ the first even $$\gamma_{2i}=\lambda_{2i-1}, \, i=1,\dots, m$$ coincides with the left end of the gap, since we know that the corresponding eigenfunction is odd. When $s=2m$ the first odd
$$\gamma_{2i-1}=\lambda_{2i}, \, i=1,\dots, m$$ coincide with the right end of the gaps since in that case $\psi_{2i}$ are odd. Note that in the limit $s\rightarrow \infty, \alpha \rightarrow 0$ such that
$s\alpha \rightarrow A/4$ we have the Mathieu operator with $u(x) = -A \cos 2x,$ so this gives us the position of all $\gamma_i$ in the gaps in that classical case.

When we apply the Darboux transformation to the cluster $\{2i-1,\, 2i\}$, as one can see from the formulas (\ref{Crum}) and (\ref{two}), that the parity of all eigenfunctions is preserved except the corresponding $\psi_{2i-1},\, \psi_{2i}$. This means that the corresponding $\gamma_i$ switched the side of the gap, while all other remain in the same position.

Thus we have
\begin{theorem} \label{invariance}
For any integer $s$ and any cluster $k$-element subset $I \subset \{1, \dots, 2m\},\, m=[s/2]$ the Schr\"odinger operator 
$$
L_I= -D^2 + 4\alpha (2k-s) \cos 2x -2\alpha^2 \cos 4x -2D^2 \log V_I(x)
$$
is a non-singular periodic operator, having the same semifinite-gap spectrum as the Whittaker--Hill operator (\ref{WHo}), but different positions of the Dirichlet eigenvalues in the gaps determined by the set $I.$ 
\end{theorem} 

Thus for any integer $s$ we have constructed $2^m-1, \, m=[s/2]$  new non-singular isospectral deformations of Whittaker--Hill equations with potentials expressible in terms of trigonometric functions. 
Some examples and graphs of the new potentials are given in Appendix 2.

A remarkable fact is that cluster sets realize all possible combinations of the positions of the Dirichlet spectrum in the open gaps, so this gives all corresponding even semifinite-gap operators. 
This implies that the converse to the Theorem \ref{theo} is also true.

\begin{corollary}
If $W_I(x)$ has no real zeros then $I$ must be a cluster subset.
\end{corollary}

Indeed, the spectral data (including the Dirichlet eigenvalues) for the corresponding potentials must coincide with one of the described above, and hence by the uniqueness theorem $I$ should be one of the cluster sets.

\section{Appendix 1: interlacing property of eigenvalues}

Recall that $n \times n$ matrix is called a {\it Jacobi matrix} if it is tri-diagonal with positive off-diagonal elements. Such matrices play an important role in the theory of orthogonal polynomials \cite{Sze}.

Let $J_n$ be an $n \times n$ Jacobi matrix
\begin{equation}
\label{jn}
J_n = 
\begin{pmatrix}
		b_0	&	c_0	&	0	&	\dots		&	\dots	&	0 	\\
		a_1	&	b_1	&	c_1	&	0		&	\dots	&	0	\\
		0	&	a_2	&	b_2	&	c_2		&	\dots	&	0	\\
		\vdots & \ddots	& \ddots 	& 	\ddots	&	\ddots	&	\vdots	\\
		0	&	\dots	&	0	&	a_{n-2}	&	b_{n-2} & c_{n-2} \\
		0	&	\dots & \dots	&	0		&	a_{n-1}	&	b_{n-1} \\
\end{pmatrix}
\end{equation}
and $J_{n-1}$ be its $(n-1)\times(n-1)$ submatrix
\begin{equation}
\label{od}
J_{n-1} = 
\begin{pmatrix}
		b_0	&	c_0	&	0	&	\dots		&	\dots	&	0 	\\
		a_1	&	b_1	&	c_1	&	0		&	\dots	&	0	\\
		0	&	a_2	&	b_2	&	c_2		&	\dots	&	0	\\
		\vdots & \ddots	& \ddots 	& 	\ddots	&	\ddots	&	\vdots	\\
		0	&	\dots	&	0	&	a_{n-3}	&	b_{n-3} & c_{n-3} \\
		0	&	\dots & \dots	&	0		&	a_{n-2}	&	b_{n-2} \\
\end{pmatrix}
\end{equation}

\begin{theorem} \label{char_int}
The eigenvalues of $J_n$ and $J_{n-1}$ are simple and interlacing. 
\end{theorem}

\begin{proof}

The proof of the simplicity of spectrum of Jacobi matrices is well-known: it follows immediately from the tri-diagonal form of the linear system $(J_n-\lambda I_n) v=0$ that for given eigenvalue $\lambda$ the corresponding eigenspace is one-dimensional.

Then by $\Delta_k(\lambda)$ we denote the characteristic polynomial for the corresponding $k \times k$ submatrix $J_k:$  
$$\Delta_k(\lambda) \equiv \det( \lambda I_k - J_k).$$ By definition we take $\Delta_0 \equiv 1.$
We are going to prove that the zeroes of $\Delta_{k+1}$ and $\Delta_{k}$ are interlacing. More precisely,
we will show that if  $\lambda_1 < \lambda_2 < \ldots < \lambda_{k+1}$ be the zeros of $\Delta_{k+1} (\lambda)$ then each interval $[\lambda_i , \lambda_{i +1}]$, $i = 1,2, \ldots, {k}$ contains exactly one zero of $\Delta_{k}(\lambda)$. 
 
We have the recurrence relation
$$
\Delta_k (\lambda) \equiv \det (\lambda I_k - J_k) = (\lambda - b_{k-1}) \Delta_{k-1}(\lambda) - c_{k-1} a_{k-2} \Delta_{k-2}(\lambda)  \qquad \textrm{for all $k \geq 2$},
$$ 
which is of the form
\begin{equation} \label{recurr}
\Delta_k (\lambda) =  (A_k\lambda + B_k) \Delta_{k-1}(\lambda) - C_k \Delta_{k-2}(\lambda)
\end{equation}
with $A_k = 1 > 0,$ $B_k=-b_{k-1}$ and $C_k = c_{k-1} a_{k-2} > 0$ since $J_k$ is Jacobi. We claim that the following identity holds
\begin{eqnarray} \label{iden1}
\lefteqn{ \gamma_0 \Delta_0(x) \Delta_0(y) + \gamma_1 \Delta_1(x) \Delta_1(y) + \ldots  + \gamma_k \Delta_k(x) \Delta_k(y) } \nonumber \\
& &  = \frac{\Delta_{k+1}(x) \Delta_k(y) - \Delta_k(x) \Delta_{k+1}(y)}{x-y} 
\end{eqnarray}
where $\gamma_i = C_{i+2} C_{i+3} \ldots C_{k+1}$ and $\gamma_k = 1$. This follows directly from the recurrence relation \eqref{recurr}:
\begin{eqnarray*}
\lefteqn{ \Delta_{k+1}(x) \Delta_k(y) -\Delta_k(x) \Delta_{k+1}(y) } \\
& = & \{ (A_{k+1}x + B_{k+1}) \Delta_{k}(x) - C_{k+1} \Delta_{k-1}(x)\} \Delta_k(y) \\
&  & -\Delta_k(x)  \{ (A_{k+1}y + B_{k+1}) \Delta_{k}(y) - C_{k+1} \Delta_{k-1}(y)\}  \\
& = &  A_k (x-y) \Delta_k(x) \Delta_k(y) + C_{k+1} \{\Delta_k(x) \Delta_{k-1}(y) - \Delta_{k-1}(x) \Delta_k(y)  \}
\end{eqnarray*}
Recalling that $A_k = 1$, this becomes
\begin{eqnarray*}
\lefteqn{ \frac{ \Delta_{k+1}(x) \Delta_k(y) -\Delta_k(x) \Delta_{k+1}(y) } {x-y}} \\
& & =  \Delta_k(x) \Delta_k(y) + C_{k+1} \frac{ \{\Delta_k(x) \Delta_{k-1}(y) - \Delta_{k-1}(x) \Delta_k(y) \}}{x-y}
\end{eqnarray*}
By induction, we obtain \eqref{iden1}. We notice that, for the special case $x=y$, we have
\begin{eqnarray*}
\lefteqn{ \gamma_0 \{\Delta_0(x)\}^2 + \gamma_1 \{ \Delta_1(x) \}^2 + \ldots  + \gamma_k \{ \Delta_k(x)\}^2 }  \\
& &  = \Delta_{k+1}'(x) \Delta_k(x) - \Delta_k'(x) \Delta_{k+1}(x)
\end{eqnarray*}
Since all $C_i>0$, it follows that for all $j$ we have $\gamma_j >0$. Hence,
\begin{equation} \label{pos}
\Delta_{k+1}'(x) \Delta_k(x) - \Delta_k'(x) \Delta_{k+1}(x) > 0
\end{equation}
because $\Delta_0 (x) \equiv 1$ and is therefore positive.

\medskip

Now, if $\xi$ and $\eta$ are two consecutive zeros of $\Delta_k(x)$, we have $\Delta_k'(\xi)\Delta_k'(\eta) < 0$. On the other hand, (\ref{pos}) yields $-\Delta_k'(\xi) \Delta_{k+1}(\xi) > 0$,  $-\Delta_k'(\eta) \Delta_{k+1}(\eta) > 0$, so that $\Delta_{k+1}(\xi) \Delta_{k+1}(\eta) <0$. This indicates an odd number of zeros of $\Delta_{k+1}(x)$ in the interval $\xi < x < \eta$. In particular, there is at least one. Now let $\xi = x_k$ be the greatest zero of $\Delta_k(x)$; then $\Delta_k'(\xi) > 0$, and (\ref{pos}) yields $\Delta_{k+1}(\xi) < 0$. Since $\Delta_{k+1}(b)$ is positive, we obtain at least one zero of $\Delta_{k+1}(x)$ on the right of $\xi = x_k$, and similarly at least one on the left of the least zero $x_1$ of $\Delta_k(x)$. Consequently, we can only have one zero of $\Delta_{k+1}(x)$ between $x_{\nu}$ and $x_{\nu +1}$, $\nu = 1,2, \ldots, k$.
By interchanging the roles of $\Delta_k(x)$ and $\Delta_{k+1}(x)$, we can prove that there exists one, and only one, zero of $\Delta_k(x)$ between every zero of $\Delta_{k+1}(x)$. 
\end{proof}

Applying this result to the matrix (\ref{ev}) we have as  a corollary that the periodic eigenvalues of Whittaker--Hill operator with $s=2m+1$ from the sets $S_0$ and $S_1$ are interlacing.

To prove a similar result for even $s$ and anti-periodic spectrum we should modify the arguments. Note that the corresponding matrices $K^{\pm}$  have the same size $m\times m$ and differ only in the first element (see (\ref{sev0})). By the Weinstein-Aronszajn
formula \cite{Kato} their characteristic polynomials $\Delta^{\pm} = \det (\lambda I_m - K^{\pm}_m)$ are related by
\begin{equation*} 
\Delta^+ =  \Delta^{-} -4\alpha s \Delta_1,
\end{equation*}
where $\Delta_1$ is the characteristic polynomial of the corresponding submatrix
\begin{equation}
\begin{pmatrix}
		b_2  &	c_2	&	0	\dots		&	\dots	&	0 	\\
		a_3	&	b_3	&	c_3	&		&	\dots	&	0	\\
				\vdots & \ddots	&    & 	\ddots	&	\ddots	&	\vdots	\\
		0	&	\dots	&	0	&	a_{m-1}	&	b_{m-1} & c_{m-1} \\
		0	&	\dots & \dots	&	0		&	a_{m}	&	b_{m} \\
\end{pmatrix}
\end{equation}
By theorem \ref{char_int} the zeros of   $\Delta^{-}$ and   $\Delta_1$ are interlacing, and hence the same is true for the zeros of $\Delta^{-}$ and   $\Delta^{+}.$ 
Thus we have

\begin{theorem}
The periodic eigenvalues of the Whittaker--Hill operator with $s=2m+1$ from the sets $S_0$ and $S_1$ are interlacing. The same is true  for the sets $S^+$ and $S^-$ for even $s$ and anti-periodic eigenfunctions.
\end{theorem}

Note that since $\alpha >0$ the set $S^-$ is shifted to the right compared to $S^+$, so the eigenvalues corresponding to the even eigenfunctions are smaller than their odd counterparts.

\section{Appendix 2: Examples and graphs}

We give first the explicit form of the Darboux transformed Whittaker--Hill potentials
$$v_I=4\alpha (2k-s) \cos 2x -2\alpha^2 \cos 4x -2D^2 \log V_I(x)$$
in the simplest cases.

The first new example comes from $s=3$ and the cluster set $I=\{0\}.$ It corresponds to the ground state
$$\psi_0=(1 + C(\alpha)\cos2x)e^{\alpha\cos2x}, \quad C(\alpha)= \frac{\sqrt{1+16\alpha^2}-1}{4\alpha}$$ 
with eigenvalue $\nu_0=2( 1 - \sqrt{1 + 16\alpha^2}).$ 
The transformed potential has the following form
$$v_0=-4\alpha \cos 2x -2\alpha^2 \cos 4x -2D^2 \log (1 + C(\alpha)\cos2x),$$
or, more explicitly,
\begin{equation} \label{first}
v_0=-4\alpha \cos 2x -2\alpha^2 \cos 4x+\frac{8 C(\alpha)(C(\alpha)+\cos 2x)}{(1 + C(\alpha)\cos2x)^2}.
\end{equation}
The corresponding Schr\"odinger operator is periodic, non-singular and has the ground state
$$\tilde\psi_0=\frac{e^{-\alpha\cos2x}}{1 + C(\alpha)\cos2x}.$$

When $s=4$ we can take cluster $I=\{1,2\}.$ The corresponding anti-periodic eigenfunctions are
$$\varphi_1= \sin x +A(\alpha) \sin 3x, \quad A= \frac{\sqrt{1-2\alpha+4 \alpha^2} +\alpha -1}{3 \alpha},$$
$$\varphi_2= \cos x +B(\alpha) \cos 3x, \quad B= \frac{\sqrt{1+2\alpha+4 \alpha^2} -\alpha -1}{3 \alpha}.$$
The corresponding wronskian (up to a sign) is
$$V_I = 1+3AB +2(A+B) \cos 2x + (B-A) \cos 4x$$
and the new potential is
\begin{equation} \label{sec}
v_{12}= -2\alpha^2 \cos 4x- 2D^2 \log [1+3AB +2(A+B) \cos 2x + (B-A) \cos 4x].
\end{equation}

The corresponding graphs as well as some of the transformed potentials in the case when $s=5$ are shown in Fig. 3 and 4.  These were created using a programme written in Maple. 

\begin{figure}
\centerline{ \includegraphics[width=6cm]{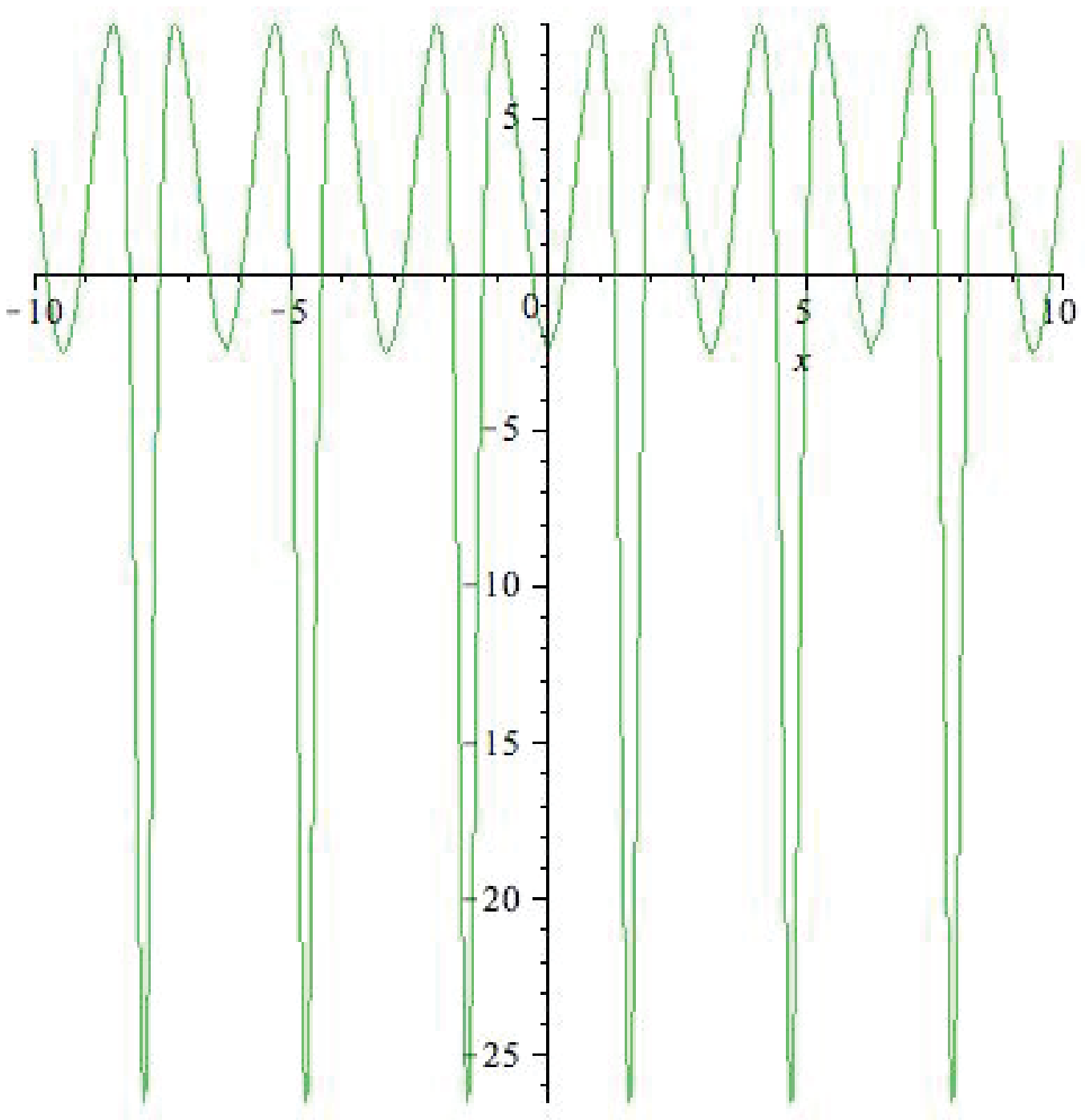} \hspace{10pt} \includegraphics[width=6cm]{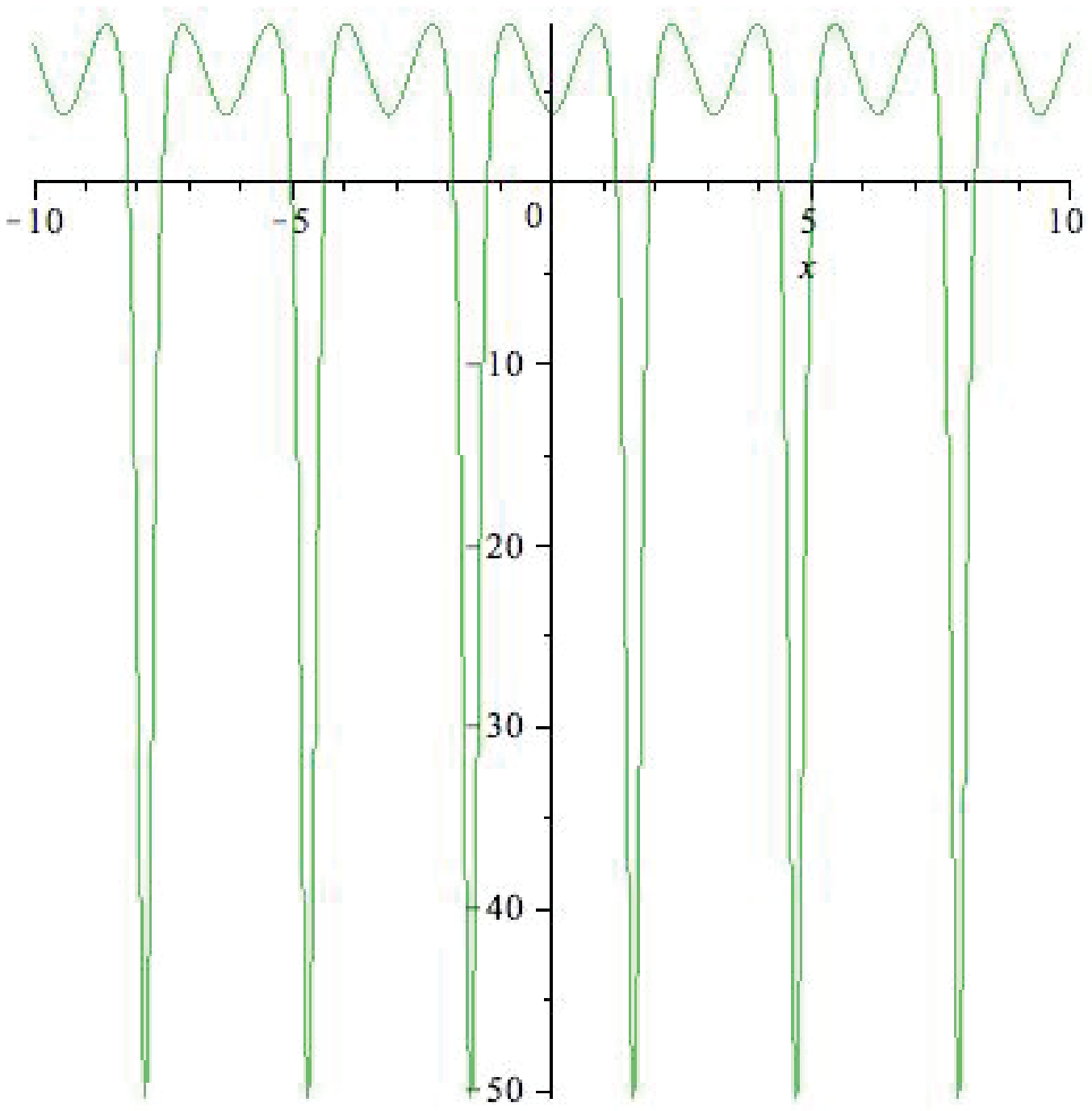} }
\caption{New semifinite-gap potentials. Left: $s=3,\,  I=\{0\}, \, \alpha=1.$ Right: $s=4, \, I=\{1,2\}, \, \alpha=1.$}\label{newex}
\end{figure}
\begin{figure}
\centerline{ \includegraphics[width=6cm]{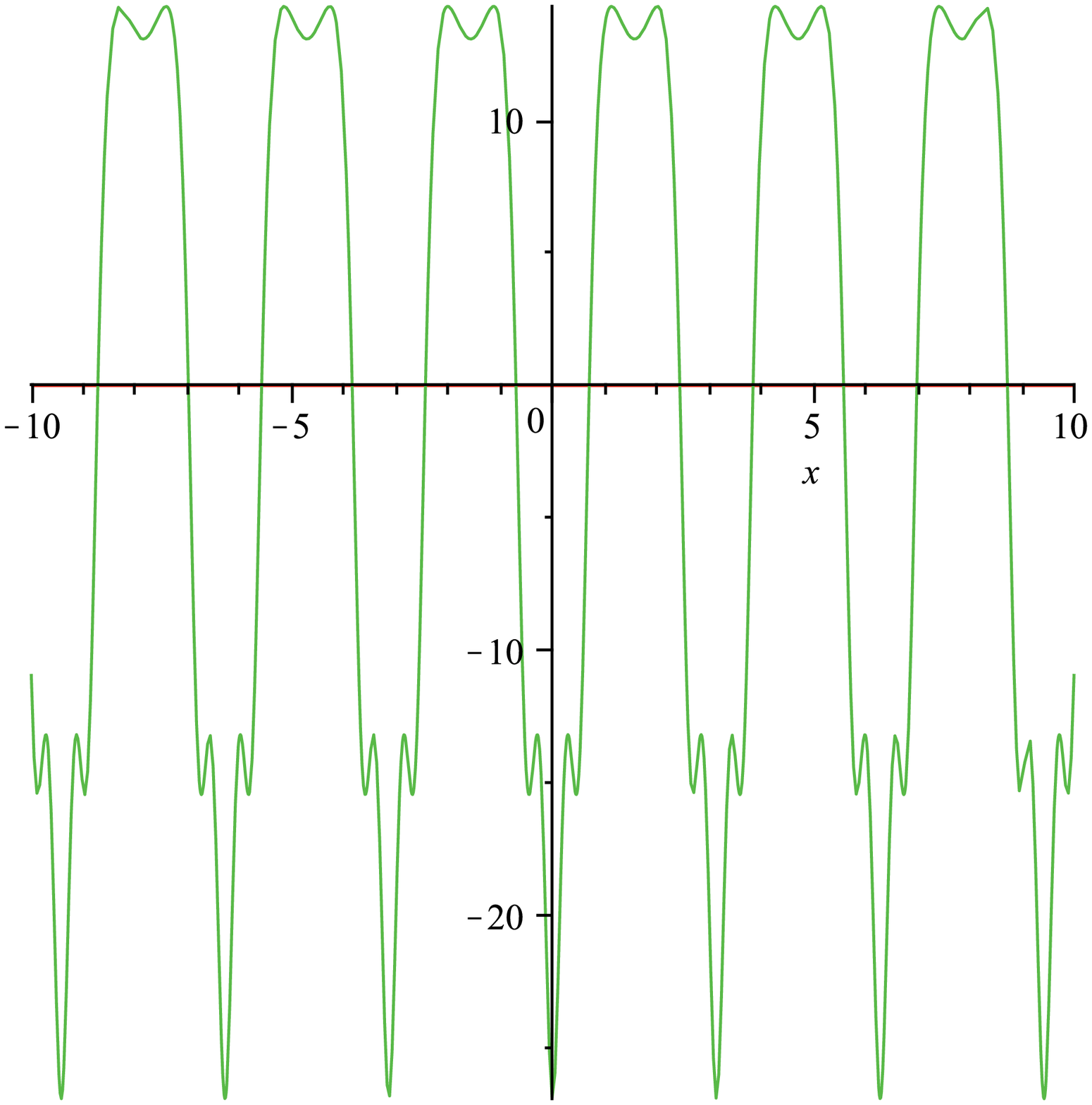} \hspace{10pt} \includegraphics[width=6cm]{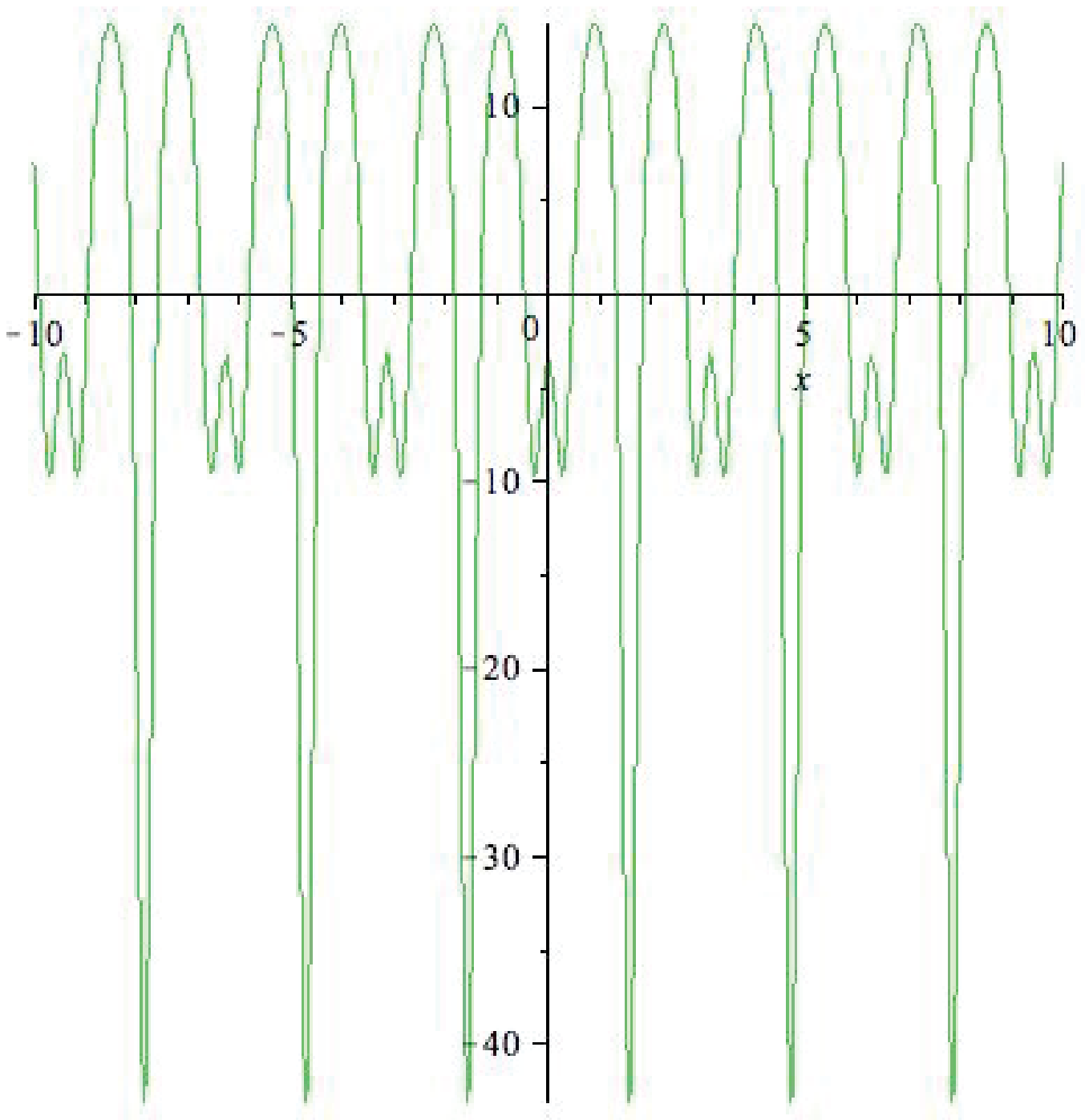} }
\caption{New semifinite-gap potentials. Left: $s=5,\,  I=\{3,4\},\,  \alpha=1.$ Right: $s=5, \, I=\{0, 3,4\}, \,\alpha=1.$} \label{semi}
\end{figure}

\section{Concluding remarks.} This paper was motivated by our attempts to study the integrability of the quantum Neumann system. In the simplest case we just have the quantum pendulum described by the 
Mathieu equation with $u(x) = A \cos 2x$. It is well-known that this equation can be considered as a limit of the Lam\`e equation with potential $u(x)=s(s+1)\wp(x),$
when the parameter $s$ goes to infinity and the elliptic curve degenerates to the rational one (see Section 15.5 in \cite{Erd}). It seems to be surprising that the elliptic deformation turned out to be simpler than the original trigonometric version. 

What we have shown in this paper (and what had actually been discovered a long time ago by Whittaker \cite{Whit}) is that one need not consider elliptic generalisations but instead simply add the second harmonic to the potential in order to see some integrability. Indeed, one can consider the Mathieu case 
as a limit of the Whittaker--Hill equation (\ref{W-H}) when $s \rightarrow \infty,\, \alpha \rightarrow 0$ such that $\alpha s \rightarrow A.$

An interesting feature of the Whittaker--Hill case is the appearance of the hierarchy of the orthogonal trigonometric polynomials. Indeed, from the orthogonality of the eigenfunctions of the Whittaker--Hill operator (\ref{WHo}) it follows for every odd $s=2m+1$ that the corresponding eigenfunctions $\varphi_k(x), \, k=0,\dots, 2m$  of the operator (\ref{K}) form a basis in the $s$-dimensional space $V_m$ of trigonometric polynomials of degree $\leq m$, which is {\it orthogonal} with respect to the scalar product
\begin{equation}
(\varphi_1, \varphi_2)= \int_0^{\pi} \varphi_1(x) \varphi_2(x) e^{2\alpha\cos2x} dx.
\end{equation} 
After the change of coordinate $z=\cos 2x$ the even eigenfunctions $\varphi_{2l}(x)=p_l(\cos 2x)$ become the usual polynomials, which are orthogonal with respect to the measure
\begin{equation}
\label{mes}
d\mu(z) = \frac{e^{2\alpha z}}{\sqrt{1-z^2}}.
\end{equation}
However, in contrast to the usual case \cite{Sze} they have {\it the same degree} $m$ if $\alpha\neq 0.$ 
When $\alpha=0$ we have the classical Chebyshev polynomials: $\varphi_{2l}(x)=\cos 2lx= T_l(\cos 2x).$ The usual orthogonal polynomials with respect to the measure (\ref{mes}) have been recently studied by Basor, Chen and Ehrhardt in \cite{BCT}. Some non-standard orthogonal polynomials are discussed by Gomez-Ullate, Kamran and Milson in the recent paper \cite{GKM}.

The situation here is similar to the paper \cite{GV}, where another quasi-exactly solvable case of sextic potential $u= x^6-(m+1)x^2$ was considered. The corresponding Darboux transformed operators have trivial monodromy in the complex domain and interesting geometry of the singular sets.
It is worthy to look at these properties in the Whittaker-Hill case as well.

\section{Acknowledgments.} We are grateful to J. Gibbons, E. Korotyaev and S.P. Novikov for stimulating discussions and useful comments.

This work has been partially supported by EPSRC (grant EP/E004008/1) and by the
European Union through the FP6 Marie Curie RTN ENIGMA (contract
number MRTN-CT-2004-5652) and through ESF programme MISGAM.

%\addcontentsline{toc}{section}{\underline{References}}
\end{document}